\documentclass{amsart}
 
\usepackage{graphicx}
 
\usepackage{amsmath,amsfonts,amssymb,amsthm}
\usepackage{mathtools}
\usepackage[all,cmtip]{xy}

\usepackage[dvipsnames]{xcolor} % Allow links to have colour. Needs to be before hyperref and tikz-cd
\usepackage[pdftex,  colorlinks=true]{hyperref} % Makes references and citations into links
    \hypersetup{urlcolor=RoyalBlue, linkcolor=RoyalBlue,  citecolor=black}
 
\usepackage[capitalise]{cleveref}
\usepackage{xcolor}

\usepackage{tikz}
\usepackage{tikz-cd}
\usetikzlibrary{positioning}

\makeatletter
\providecommand\@dotsep{5}
\def\listtodoname{List of Todos}
\def\listoftodos{\@starttoc{tdo}\listtodoname}
\makeatother

\newtheorem{theorem}{Theorem}[section]
\newtheorem{proposition}[theorem]{Proposition}

\newtheorem{lemma}[theorem]{Lemma}

\newtheorem{thmx}{Theorem}
\newtheorem{corx}[thmx]{Corollary}
\newtheorem{propx}[thmx]{Proposition}
\newtheorem{quex}[thmx]{Question}

  \theoremstyle{definition}
\newtheorem{definition}[theorem]{Definition}

\newtheorem{remark}[theorem]{Remark}

\newcommand{\nclose}[1]{\ensuremath{\langle\!\langle#1\rangle\!\rangle}}

\newcommand{\dbN}{\mathbb{N}}
\newcommand{\dbQ}{\mathbb{Q}}

\newcommand{\dbZ}{\mathbb{Z}}

\newcommand{\calF}{{\mathcal F}}

\newcommand{\calG}{{\mathcal G}}
\newcommand{\calH}{{\mathcal H}}

\newcommand{\calO}{{\mathcal O}}
\newcommand{\calP}{{\mathcal P}}

\newcommand{\calR}{{\mathcal R}}

\newcommand{\orf}[1]{O_\calF #1}

\newcommand{\vcyc}{V\text{\tiny{\textit{CYC}}}}
\newcommand{\fin}{F\text{\tiny{\textit{IN}}}}

\newcommand{\evc}{\underline{\underline{E}}}
\DeclareMathOperator{\cd}{cd}

\DeclareMathOperator{\cdfin}{\underline{cd}}
\DeclareMathOperator{\cdvc}{\underline{\underline{cd}}}
\DeclareMathOperator{\dist}{\mathsf{dist}}

\DeclareMathOperator{\gd}{gd}
\DeclareMathOperator{\gdfin}{\underline{gd}}
\DeclareMathOperator{\gdvc}{\underline{\underline{gd}}}

\DeclareMathOperator{\Aut}{Aut}

\begin{document}

\title[ ]{Bowditch Taut Spectrum and dimensions of groups}

\author[E. Martínez-Pedroza]{Eduardo Martínez-Pedroza}
\address{Memorial University  St. John's, Newfoundland and Labrador, Canada}
\email{eduardo.martinez@mun.ca}

\author[L.J. Sánchez Saldaña]{Luis Jorge S\'anchez Salda\~na}
\address{Departamento de Matemáticas,
Facultad de Ciencias, Universidad Nacional Autónoma de México, Mexico}
\email{luisjorge@ciencias.unam.mx}
 
\subjclass[2020]{Primary 57M07, 57M60, 20F65, 55R35}

\date{}

\keywords{Taut spectrum, quasi-isometry classes, classifying spaces, geometric and cohomological dimensions}

\begin{abstract}
For a finitely generated group $G$, let $H(G)$ denote Bowditch's taut loop length spectrum. We prove that if $G=(A\ast  B) / \langle\!\langle \mathcal R \rangle\!\rangle $ is a $C'(1/12)$ small cancellation quotient of a the free product of finitely generated groups, then $H(G)$ is equivalent to $H(A) \cup H(B)$. 
We use this result together with bounds for cohomological and geometric dimensions, as well as Bowditch's construction of continuously many non-quasi-isometric $C'(1/6)$ small cancellation $2$-generated groups  to obtain our main result: Let $\mathcal{G}$ denote the class of finitely generated groups. The following subclasses contain  continuously many  one-ended non-quasi-isometric groups:
\begin{enumerate}
     \item  $\left\{G\in \mathcal{G} \colon \underline{\mathrm{cd}}(G) = 2 \text{  and } \underline{\mathrm{gd}}(G) = 3 \right\}$
  
     \item $\left\{G\in \mathcal{G} \colon \underline{\underline{\mathrm{cd}}}(G) = 2 \text{  and } \underline{\underline{\mathrm{gd}}}(G) = 3 \right\}$

    \item $\left\{G\in \mathcal{G} \colon \mathrm{cd}_{\mathbb{Q}}(G)=2 \text{  and } \mathrm{cd}_{\mathbb{Z}}(G)=3 \right\}$
   \end{enumerate}
   
   On our way to proving the aforementioned results, we show that the classes defined above are closed under taking relatively finitely presented $C'(1/12)$ small cancellation quotients of free products, in particular, this produces new examples of groups exhibiting an Eilenberg-Ganea phenomenon for families.
   
   We also show that if there is a finitely presented counter-example to the Eilenberg-Ganea conjecture, then  there are  continuously many finitely generated one-ended non-quasi-isometric counter-examples.
\end{abstract}

\maketitle

\setcounter{tocdepth}{1}
\tableofcontents

\section{Introduction}

Let $G$ be a discrete group. Let $\cd_R(G)$ denote the cohomological dimension of $G$ with respect to the ring $R$. Analogously, let $\cdfin(G)$ and $\cdvc(G)$ denote the proper cohomological dimension and the virtually cyclic cohomological dimension of $G$ respectively, and let $\gdfin(G)$ and $\gdvc(G)$ be their geometric counterparts. For explicit definitions see Section~\ref{sec:dimensional:bounds}. The main result of this  manuscript is the following. 

\begin{thmx}\label{thm:main}
Let $\mathcal{G}$ denote the class of finitely generated groups. The following subclasses contain  continuously many  one-ended non-quasi-isometric groups:
\begin{enumerate}
     \item  $\left\{G\in \mathcal{G} \colon \cdfin(G) = 2 \text{  and } \gdfin(G) = 3 \right\}$
  
     \item $\left\{G\in \mathcal{G} \colon \cdvc(G) = 2 \text{  and } \gdvc(G) = 3 \right\}$

    \item $\left\{G\in \mathcal{G} \colon \cd_{\mathbb{Q}}(G)=2 \text{  and } \cd_{\mathbb{Z}}(G)=3 \right\}$
   \end{enumerate}
\end{thmx}

The proof of the theorem uses a quasi-isometry invariant of finitely generated groups,
introduced by Bowditch~\cite{Bow98}, known as the \emph{taut loop length spectrum} (or \emph{taut spectrum} for short). For a finitely generated group $G$ this invariant,  denoted as $H(G)$,  takes values on the power set of the natural numbers modulo an  equivalence relation, see Section~\ref{sec:Taut}. Bowditch observed that if $A$ and $B$ are finitely generated quasi-isometric groups, then the taut spectrums $H(A)$ and $H(B)$ are equivalent~\cite[Lemma~3]{Bow98}.
The taut spectra of all finitely presented groups are equivalent. In contrast, using small cancellation theory, Bowditch proved the following result on which our main result relies on. 

\begin{thmx}\cite{Bow98}\label{thm:bowditch:B:alpha}
There are continuously many  torsion-free $2$-generator $C'(1/6)$ small cancellation groups $\{B_\alpha \colon \alpha\in \mathbb{R}\}$ such that $H(B_\alpha)$ and $H(B_\beta)$ are equivalent only if $\alpha=\beta$.
\end{thmx}

The second ingredient to prove the main result is the following computation of the taut spectrum for certain small cancellation quotients of free products, as defined in the book by Lyndon and Schupp~\cite[Chapter  V.11]{LySc01}.
Recall that a group of the form $G = (A\ast_C  B)/ \nclose{R}$ where $\mathcal{R}$ is a  symmetrized subset of $A\ast_C  B$ satisfying the $C'(1/6)$ small cancellation condition is called a \emph{small cancellation product}. If $\mathcal R$ is finite, we say $G$ is  \emph{relatively finitely presented}. If no element of $\mathcal{R}$ is a proper power, we say that \emph{$G$ is relatively torsion-free}.

\begin{thmx}[\Cref{thm:taut:small:cancellation}]\label{thm:TautProducts}
Let $A$ and $B$ be finitely generated groups. Let $\mathcal R$ be a finite symmetrized subset of $A\ast  B$ that satisfies the $C'(1/6)$ small cancellation condition, and let $G = (A\ast  B)/ \nclose{\mathcal R}$.
Then 
$H(G)$ and $H(A) \cup H(B) $ are equivalent.
\end{thmx}

The third ingredient to prove our main result is a collection of dimensional bounds for multi-ended groups. We say a group $G$ \emph{has small centralizers} if every infinite cyclic subgroup $C$ of $G$ has finite index in its centralizer $C_G(C)$.

\begin{thmx}[Theorems~\ref{cor:closed:under:takings:graphs} and \ref{cor:megafinal-body} ]\label{cor:closed:under:takings:graphs:intro}
Let $G$ be the fundamental group of a graph of groups with finite edge stabilizers, and let $T$ be the Bass-Serre tree.
Then the following inequalities hold
\[\gdfin(G)\leq \max\{1,\gdfin(G_{\sigma})| \ 
\sigma\in \ T \},\quad \cdfin(G)\leq \max\{1,\cdfin(G_{\sigma})| \ 
\sigma \in \ T \}\]
\[\gdvc(G)\leq \max\{2,\gdvc(G_{\sigma})| \ 
\sigma \in \ T \},\quad\cdvc(G)\leq \max\{2,\cdvc(G_{\sigma})| \ 
\sigma \in \ T \}.\]
If, in addition, all vertex groups have small centralizers and satisfy the ascending chain condition for finite subgroups, then 
\[\gdvc(G)\leq \max\{2,\gdfin(G_{\sigma})| \ 
\sigma \in \ T \}\text{ and }\cdvc(G)\leq \max\{2,\cdfin(G_{\sigma})| \ 
\sigma \in \ T \}.\]

\end{thmx}

The fourth ingredient is a computation of cohomological and geometric dimensions  for small cancellation products. It could be interpreted as a generalization of  previous result in the sense that some of those dimensional bounds are preserved under large quotients.

\begin{thmx}[\cref{pre-Machine-body}]\label{pre-Machine} 
Let $A$ and $B$ be groups and let $C$ be a common finite subgroup.  Let $\mathcal R$ be a finite symmetrized subset of $A\ast_C B$ that satisfies the $C'(1/12)$ small cancellation condition, and $G = (A\ast_C B)/ \nclose{\mathcal R}$. Assume $G$ is not virtually free. Then the following conclusions hold.
\begin{enumerate}
    \item $ \gdfin(G) =\max\{ \gdfin(A),\gdfin(B), 2\} \text{ and }   \cdfin(G) =\max\{ \cdfin(A),\cdfin(B), 2\}$
    
     \item If  $A$ and $B$ are finitely generated, have small centralizers and satisfy the  ascending chain condition for finite subgroups, then $G$ has small centralizers, satisfies the ascending chain
    condition, and
     \[\max\{\gdvc(A),\gdvc(B),2\}\leq \gdvc(G)\leq\max\{\gdfin(A),\gdfin(B),2\}\]
     and
    \[\max\{\cdvc(A),\cdvc(B),2\}\leq \cdvc(G)\leq\max\{\cdfin(A),\cdfin(B),2\}.\]
    
    \item For any ring $R$, we have $ \cd_{R}(G) =\max\{ \cd_{R}(A),\cd_{R}(B), 2\} .$ 
\end{enumerate}
\end{thmx}

A notable hypothesis of \Cref{pre-Machine} is the $C'(1/12)$ small cancellation condition instead of the standard $C'(1/6)$ condition. While we do not claim that $C'(1/12)$  is the optimal hypothesis, we refer the reader to \Cref{rem:SmallCancellation} which gives an insight of the difficulty. 

The fifth ingredient to prove the main result are the following statements that provide sufficient conditions for certain groups to be one-ended.

\begin{thmx}[\cref{prop:OneEndedProducts}]\label{prop:OneEndedProducts:intro}
Let $G=(A\ast_C B)/\langle\langle \mathcal R \rangle\rangle$, where $C$ is a common finite subgroup of $A$ and $B$, and $\mathcal R$ is a symmetrized subset of $A\ast B$ that satisfies the $C'(1/6)$ small cancellation condition.
\begin{itemize}
\item Suppose that for  every $r\in \mathcal{R}$, its normal form does not  contain   elements of  finite order of $A$ or $B$.
\end{itemize}
If $A$ and $B$ are one-ended, then $G$ is one-ended.
\end{thmx}
In \cref{prop:OneEndedProducts:intro},   we do not know whether the bulleted hypothesis is necessary.  We also use the following proposition, that was communicated to us by Dani Wise, to prove our main result.
\begin{propx}[\cref{prop:wise2}]\label{prop:wise2:intro}
Let $G$ be a torsion-free, $2$-generated group. If $G$ is not a free group, then $G$ is one-ended.
\end{propx}

Theorem~\ref{thm:main} will follow from the previous stated results;    the fact that the classes of finitely generated groups in the statement are  non-empty by results of Brady, Leary and Nucinkis~\cite{BLN01}, Fluch and Leary~\cite{FL14}, and Bestvina and Mess~\cite{BeMe91} respectively; and the following abstraction.

\begin{thmx}\label{thm:abstract}
Let $\mathrm{Q} \subset \mathrm{P}$ be classes of finitely generated groups that are closed under taking finitely presented  subgroups. Suppose
\begin{enumerate}
    \item $\mathrm{P}$ contains the class of  torsion-free  finitely generated $C'(1/6)$ small cancellation groups.
    \item $\mathrm{P}$ is closed under taking relatively finitely presented $C'(1/6)$ small cancellation products,
    \item $\mathrm{Q}$ is closed under taking  amalgamated products over finite subgroups and HNN extensions over finite subgroups, and
\item   there is a hyperbolic group $G$ that is in $\mathrm{P}$ but not in $\mathrm{Q}$.
\end{enumerate}
Then there are continuously many one-ended non-quasi-isometric groups $\mathrm{P}$ that are not in $\mathrm{Q}$.
\end{thmx}
\begin{proof}
First we argue that  there is a group $A$ in $\mathrm{P}$ that is not $\mathrm{Q}$ which is one-ended, finitely presented, and contains   elements of infinite order. This is a direct consequence Dunwoody's  accessibility~\cite{Dun85}. Since $G$ is hyperbolic, it is finitely presented and hence it splits as the fundamental group of a graph of groups with one-ended finitely presented vertex groups and finite edge groups. In fact,  each vertex group is a quasi-convex subgroup of $G$, and hence it is  hyperbolic. By hypothesis on $\mathrm{P}$, all vertex groups are in $\mathrm{P}$, but not all can be in  $\mathrm{Q}$, else $G$ would belong to $\mathrm{Q}$ by hypothesis (4), which is a contradiction. Therefore there is a one-ended hyperbolic group $A$ that is $\mathrm{P}$ but not in $\mathrm{Q}$, and such a group always contains elements of infinite order.

By \cref{thm:bowditch:B:alpha}, there is a continuous  collection of groups $\{B_\alpha \colon \alpha \in \mathbb{R}\}$ where each $B_\alpha$ is a torsion-free $2$-generated $C'(1/6)$ small cancellation group such that $B_\alpha$ is not quasi-isometric to $B_\beta$ if $\alpha\neq \beta$. In  particular, each $B_\alpha$ is a one-ended group (see \cref{prop:wise2:intro}).  By the first assumption on $\mathrm{P}$,  $B_\alpha \in \mathrm{P}$ for each $\alpha\in\mathbb{R}$. 

For each $\alpha\in \mathbb{R}$, let $\mathcal{R}_\alpha$ be a finite symmetrized subset  of $A\ast B_\alpha$ that satisfies the $C'(1/12)$ small cancellation condition and such that each $r\in \mathcal{R}_\alpha$ only involves elements of infinite order of $A$ and $B_\alpha$. For example, if $a\in A$ and $b_1,b_2\in B_\alpha$ are distinct elements of infinite order,  let $\mathcal{R}_\alpha$ be the symmetrized subset generated by $r\in A\ast B_\alpha$ given by  
\[ r=  (ab_1)ab_2(ab_1)^2ab_2(ab_1)^3\cdots ab_2(ab_1)^{12}ab_2.\] 

Let $G_\alpha$ be the quotient group $(A \ast B_\alpha)/ \nclose{\mathcal{R}}$.
By the second assumption on $\mathrm{P}$, the group $G_\alpha \in \mathrm{P}$ for each $\alpha \in \mathbb{R}$.
Since $G_\alpha$ contains  a subgroup isomorphic to $A$ (see for example~\cite[Ch. V. Cor. 9.4]{LySc01}), and $A\not\in \mathrm{Q}$, it follows that $G_\alpha \not\in \mathrm{Q}$. Since $A$ and $B_\alpha$ are one-ended, it follows that 
$G_\alpha$ is one-ended as well (see \cref{prop:OneEndedProducts:intro}). Since $A$ is finitely presented, Theorem~\ref{thm:TautProducts} implies that $H(G_\alpha)$ is equivalent to $H(B_\alpha)$ and hence $G_\alpha$ and $G_\beta$ are  quasi-isometric only if $\alpha= \beta$.
\end{proof}

We are now ready to prove  our main result:

\begin{proof}[Proof of Theorem~\ref{thm:main}]
For the first item, let $\mathrm{P}$ be the class of finitely generated groups $G$ such that $\cdfin(G)\leq 2$ and $\gdfin(G)\leq 3$. Let $\mathrm{Q}$ be the subclass of $\mathrm{P}$ defined by groups $G$ with $\gdfin G \leq 2$. Note that both $\mathrm{P}$ and $\mathrm{Q}$ are closed under taking finitely presented subgroups.  Since torsion-free $C'(1/6)$ small cancellation groups have geometric dimension at most two, they belong to $\mathrm{P}$. By \cref{pre-Machine}(1), the class $\mathrm{P}$ is closed under taking $C'(1/12)$ small cancellation products. From \cref{cor:closed:under:takings:graphs:intro} we get that $\mathrm{Q}$ is closed under taking amalgamated products  over finite groups and HNN extensions over finite subgroups.  Note that a group $G$  is in $\mathrm{P}$ but not in $\mathrm{Q}$ if and only if $\cdfin(G)=2$ and $\gdfin(G)=3$.
By a result  of Brady, Leary and Nucinkis~\cite{BLN01}, there is a hyperbolic group $\Gamma$ in $\mathrm{P}$ that is not in $\mathrm{Q}$.    Then Theorem~\ref{thm:abstract} implies the first statement of the theorem. 

For the second statement of the theorem, let $\mathrm{P}$ be the class of finitely generated groups $G$ such that   
\begin{itemize}
    \item $G$ has the small centralizers property, that is, the centralizer of any infinite cyclic subgroup of $G$ is virtually cyclic,  
    \item $G$ satisfies the ascending chain condition for finite subgroups, and
    \item $\cdvc(G)\leq 2$, $\gdvc(G)\leq 3$, $\cdfin(G)\leq 2$, and $\gdfin(G)\leq 3$.
\end{itemize}
Define $\mathrm{Q}$ as the subclass of groups $G$ that additionally satisfy $\gdvc(G)\leq 2$. It is not difficult to see that both $\mathrm{P}$ and $\mathrm{Q}$ are closed under taking subgrups.

Any torsion-free finitely generated $C'(1/6)$ small cancellation group $G$ 
has the small centralizers property, see~\cite{Truf74} or~\cite{Sey74}, and trivially satisfies the ascending chain condition for finite subgroups. Now, $G$ has geometric dimension two, so $\cdfin(G)\leq \gdfin(G)\leq 2$. Since $G$ has small centralizers and satisfies the ascending chain condition,
\[\gdvc(G)\leq \max\{2,\gdfin(G)\} \text{ and } \cdvc(G)\leq \max\{2,\cdfin(G)\},\]
see \cref{lemma:bound:dimensions}. Hence $\cdvc(G)\leq 2$ and $\gdvc(G)\leq 3$, and therefore  $G$ belongs to the class $\mathrm{P}$. That $\mathrm{P}$ is closed under taking $C'(1/6)$ small cancellation products is a direct consequence of \cref{pre-Machine}(2). A result of Fluch and Leary~\cite{FL14} shows that there is a hyperbolic group $\Gamma$ which satisfies $\cdfin(\Gamma)=2$, $\gdfin(\Gamma)=3$, $\cdvc(\Gamma)=2$, and $\gdvc(\Gamma)=3$. It is well-known that hyperbolic groups have small centralizers and satisfy the ascending chain condition for finite subgroups, see for example \cite{BH99}. Thus $\Gamma$ is in $\mathrm{P}$ but not in $\mathrm{Q}$. By \cref{cor:closed:under:takings:graphs:intro}, the class $\mathrm{Q}$ is closed under taking amalgamated products under over finite groups and HNN extensions over finite subgroups. Therefore the second statement follows from  \cref{thm:abstract}.

For the last statement, let $\mathrm{P}$ be the class of finitely generated groups $G$ such that $\cd_\mathbb{Q}(G)\leq 2$ and $\cd_\mathbb{Z}(G)\leq 3$.
Let $\mathrm{Q}$ be the subclass of $\mathrm{P}$ defined by groups $G$ with $\gd_\mathbb{Z}(G) \leq 2$. Then the argument follows from Theorem~\ref{thm:abstract} in an analogous way as in the previous paragraphs, in this case, invoking Theorem~\ref{pre-Machine}(3) and the existence of a hyperbolic group $G$ such that $\cd_\mathbb{Q}(G)=2$ and $\cd_\mathbb{Z}(G)=3$. The existence of such a group was described by Bestvina and Mess~\cite[Paragraph before Corollary~1.4]{BeMe91} and it uses their dimension formula and  a technique of Gromov and Davis-Januszkiewicz~\cite{DJ91}; there are other examples described by Dranishnikov~\cite[Corollary 2.3 ]{D99}.
\end{proof}

\subsection*{Other results.}

From the proof of \cref{thm:abstract}, observe that  hypothesis (4) of this theorem could be replaced with \emph{there is a finitely presented torsion-free group that is in $\mathrm{P}$ but not in $\mathrm{Q}$}. Hence the same type of argument proving \cref{thm:main} proves the following statement on the Eilenberg-Ganea conjecture~\cite{EG57}. 

\begin{corx}
If the class of groups  $\left\{G\in \mathrm{Groups} \colon \cd(G) = 2 \text{  and } \gd(G) = 3 \right\}$ contains a finitely presented group, then  it contains continuously many finitely generated one-ended non-quasi-isometric groups.
\end{corx}

We record the following corollary which is a direct consquence of Theorem~\ref{pre-Machine}.  This result describes a way to construct new groups  that exhibit an Eilenberg-Ganea phenomenon for the families of finite and virtually cyclic subgroups, as well as (hyperbolic) groups of rational cohomological dimension two and integral cohomological dimension three. Let $G$ be a group and $\calF$ be a family of subgroups of $G$. We say that $G$ is an \emph{$\calF$-Eilenberg-Ganea group} if $\gd_\calF(G)=3$ and $\cd_\calF(G)=2$.

\begin{corx}\label{Machine}
Let $A$ and $B$ be groups.  Let $\mathcal R$ be a finite symmetrized subset of $A\ast B$ that satisfies the $C'(1/12)$ small cancellation condition, and $G = (A\ast B)/ \nclose{\mathcal R}$. 
\begin{enumerate}
    \item If  $A$ is a $\fin$-Eilenberg-Ganea group,  and  $\cdfin(B) \leq 2$, then $G$ is a $\fin$-Eilenberg-Ganea group. 
    \item If $A$ is an $\calF$-Eilenberg-Ganea group for both $\calF=\fin$ and $\vcyc$, $\cd_{\calF}(B) \leq 2$ for both $\calF=\fin$ and $\vcyc$, and $A$ and $B$ are finitely generated, have small centralizers, and satisfy the chain ascending condition for finite subgroups, then $G$ is an $\vcyc$-Eilenberg-Ganea group.
    
    \item If $\cd_{\dbZ}(A) = 3$, $\cd_{\dbQ}(A) = 2$, $\cd_{\dbZ}(B) \leq 3$, and $\cd_{\dbQ}(B) \leq 2$, then $\cd_{\dbZ}(G) = 3$ and $\cd_{\dbQ}(G) = 2$.  
\end{enumerate}
If, in addition, $A$ and $B$ are one-ended and $\mathcal{R}$ satisfies the hypothesis of \cref{prop:OneEndedProducts:intro}, then $G$ is one-ended. Similarly if $A$ and $B$ are finitely presented, then $G$ is finitely presented.
\end{corx}

%We also  prove the following dimensional bounds that are of independent interest, compare with \cref{cor:closed:under:takings:graphs:intro}. \begin{thmx}\label{cor:megafinal}Let $G$ be the fundamental group of a graph of groups with finite edge stabilizers, and let $T$ be the Bass-Serre tree. Asssume all vertex groups have small centralizers and satisfy the ascending chain condition for finite subgroups.Then \[\gdvc(G)\leq \max\{2,\gdfin(G_{\sigma})| \ \sigma \in \ T \}\text{ and }\cdvc(G)\leq \max\{2,\cdfin(G_{\sigma})| \ \sigma \in \ T \}.\]\end{thmx}

\subsection*{A couple of questions}
Since the taut spectrum of any two finitely presented groups are equivalent, it can only distinguish between quasi-isometry types of infinitely presented groups. Thus the following question is out of the scope of the present article.

\begin{quex}
For each of the classes in Theorem~\ref{thm:main}, does the class contain infinitely many non-quasi-isometric  finitely presented groups? 
\end{quex}

Let $n\geq0$ be an integer. A group is said to be \emph{virtually $\mathbb{Z}^n$} if it contains a subgroup of finite index isomorphic to $\mathbb{Z}^n$. For any group $G$, define the family \[\calF_n=\{H\leq\Gamma| H \text{ is virtually } \mathbb{Z}^r \text{ for some 
} 0\le r \le n \}.\] 
Observe that $\calF_0=\fin$ is the family of finite subgroups and $\calF_1=\vcyc$ is the family of virtually cyclic subgroups, these  are particularly relevant due to their connection with the Farrell-Jones and Baum-Connes isomorphism conjectures~\cite{FJ93,Lu05}.

\begin{quex}
Does the class of finitely generated groups \[\left\{G\in \mathcal{G} \colon \cd_{\calF_k}(G) = 2 \text{  and } \gd_{\calF_k}(G) = 3 \right\}\] satisfy the conclusion of Theorem~\ref{thm:main}? This class is non-empty by~\cite[Rem.~4.6]{SS20}.
\end{quex}

\subsection*{Organization} The rest of article  proves the results stated above. Each section corresponds to one of the results. 

Section~\ref{sec:Taut} is on the Taut spectrum of small cancellation quotients of free products. It contains an argument  that shows that the taut spectrum of a  free product  of finitely generated groups is equivalent to the union of the spectrums of the factors, $H(A\ast B) = H(A)\cup H(B)$, see \Cref{prop:TautProduct}. It concludes with  the proof of Theorem~\ref{thm:TautProducts}.

Section~\ref{sec:dimensional:bounds} is on the  proper and virtually cyclic dimensions of multi-ended groups, and in particular, it discusses the proof of  \Cref{cor:closed:under:takings:graphs:intro}.

Section~\ref{sec:Thn1.5} uses some of the results of the previous section to obtain upper bounds on the proper and virtually cyclic dimensions of small cancellation products of free groups, and in particular discusses  the proof of Theorem~\ref{pre-Machine}.

Section~\ref{sec:SmallCancellation} is on one-ended small cancellation quotients of free products, and 
contains the  proofs of \cref{prop:OneEndedProducts:intro} and \cref{prop:wise2:intro}.

\subsection*{Acknowledgements} We thank the referee for comments, suggestions and corrections that improved the quality of the manuscript. We also thank  Brita Nucinkis for pointing out the reference~\cite{FN13}, and  Dani Wise for communicating the argument proving~\Cref{prop:wise2}. Both authors wish to thank Sam Hughes for several comments and proof reading of the manuscript. The first author acknowledges funding by the Natural Sciences and Engineering Research Council of Canada, NSERC. The second author was supported by grant PAPIIT-IA101221.

\section{Taut spectrum of small cancellation quotients of free products} \label{sec:Taut}

\subsection{The taut loop length spectrum}
Let us recall the definition of Bowditch's  \emph{taut loop length spectrum} (or \emph{taut spectrum} for shorter) introduced in \cite{Bow98}, using the approach described in \cite{KLS20}. 

Let $\Gamma$ be a connected, simplicial graph. An \emph{edge loop of length $l$} in  $\Gamma$ is a sequence of $v_0,\dots ,v_l$ of vertices such that $v_0=v_l$ and $\{v_{i-1},v_i\}$ is an edge for $1\leq i\leq l$. For a graph $\Gamma$ and a fixed integer $l$, let $\Gamma_l$ denote the $2$-complex whose $1$-skeleton is the geometric realization of $\Gamma$, with one $2$-cell attached to each edge loop in $\Gamma$ of length strictly less than $l$. An edge loop of length $l$ is said to be \emph{taut} if it is not null-homotopic in $\Gamma_l$. The \emph{taut loop length spectrum} of $\Gamma$ is by definition
\[H(\Gamma)=\{n\in \dbN | \text{ there is a taut loop in $\Gamma$ of length $n$}\}.\]
As observed  in~\cite[Lemma~2.2]{KLS20},  an equivalent definition of the taut  loop length spectrum of $\Gamma$ is
\[H(\Gamma)=\{l\in \dbN|\ \pi_1(\Gamma_l)\to \pi_1(\Gamma_{l+1}) \text{ is not an isomorphism} \}\]
where the map between fundamental groups is the one induced by inclusion.

The vertex set of the graph $\Gamma$ is endowed with the combinatorial path metric $d_\Gamma$, that is, the distance between two vertices is the length of the shortest edge path between them.  For $k>0$, a function $f\colon X\to Y$ between metric spaces  is \emph{k-Lipschitz} if $d_Y(f(x),f(x'))\leq kd_X(x,x')$ for all $x,x'\in X$. Two graphs $\Gamma$ and $\Lambda$ are \emph{$k$-quasi-isometric} if there exist a pair of $k$-Lipschitz maps between vertex sets $\phi\colon V(\Gamma)\to V(\Lambda)$ and $\psi\colon V(\Lambda)\to V(\Gamma)$ such that $d_\Gamma(x,\psi\circ \phi(x))\leq k$ and $d_\Gamma(y,\phi\circ \psi(y))\leq k$ for all $x\in V(\Gamma)$ and $y\in V(\Lambda)$.  Two simplicial connected graphs are \emph{quasi-isometric} if their vertex sets with the combinatorial path metrics are $k$-quasi-isometric metric spaces for some $k>0$.

Two subsets $H$  and $H'$ of $\dbN$ are said to be $k$-related, and write $H\sim_k H'$, if for all $l\geq k^2+2k+2$, whenever $l\in H$ then there exists $l'\in H'$ such that $l/k\leq l'\leq lk$ and vice-versa.  Two subsets of $\mathbb{N}$ are equivalent if they are $k$-related for some $k$.

\begin{lemma}\cite[Lemma~3]{Bow98}
If two connected simplicial graphs $\Gamma$ and $\Lambda$ are $k$-quasi-isometric, then $H(\Gamma)$ and $H(\Lambda)$ are $k$-related.
\end{lemma}

Let $G$ be a group and $S$ a finite generating set. Then we denote $H(G,S)=H(\Gamma(G,S))$ where $\Gamma(G,S)$ is the Cayley graph of $G$ with respect to $S$. Hence the taut spectrum of a group $G$ is well defined up to equivalence. If the finite generating set  of $G$ is clear from the context we just write $H(G)$ for the taut spectrum.

\subsection{The taut spectrum of a free product}

This subsection computes the taut spectrum of a free product. It is a warm-up  for the next subsection where we compute the taut spectrum of certain small cancellation quotients of free products.

\begin{proposition}\label{prop:TautProduct}
Let $A$ and $B$ be groups with finite generating sets $S_A$ and $S_B$ respectively. Then
\[H(A\ast B,S_A\sqcup S_B)=H(A,S_A)\cup H(B,S_B).\]
\end{proposition}
\begin{proof}
A cyclic word $w$ in the alphabet $S_A\cup S_B$ represents the identity in $G$ only if there is a syllable (a maximal subword in $S_A$ or $S_B$) of $w$ that represents the identity.
Therefore if $w$ is a reduced cyclic word of length $m$ that represents the identity then either
\begin{enumerate}
    \item $w$ is an $A$-word or a $B$-word; or
    \item $w$ is equal  to a product $\prod r_i^{g_i}$ in $G$ where each $r_i$ is an $A$-relation or $B$-relation of length strictly less than $m$. 
\end{enumerate}

In the Cayley graph $\Gamma(G)$, a cycle is called homogeneous if its label is either an $A$-word or a $B$-word. Let $\Gamma_\ell'(G)$ be the complex with $1$-skeleton the Cayley graph of $G$ and a $2$-cell for each homogenous cycle of length less than $\ell$. A consequence of the   above observations is that the inclusion
\[ \Gamma_\ell'(G) \hookrightarrow \Gamma_\ell(G) \]
 is $\pi_1$-injective (and hence $\pi_1$-isomorphism). Indeed, if $\gamma$ is a cycle of length less than $\ell$ in $\Gamma_\ell(G)$ then there is disk diagram in $\Gamma'_\ell(G)$ for $\gamma$. 
 
 Let $\ell>0$ and consider the inclusion
 \[ \Gamma_\ell(B) \to \Gamma_\ell'(G).\]
 Let us show   that this inclusion is $\pi_1$-injective. 
 Let $\gamma\colon S^1\to \Gamma_\ell(B)$ be a closed path labeled by a $B$-word, and suppose that  $\gamma$ that is null homotopic in $\Gamma_\ell'(G)$. We claim that $\gamma$ is null-homotopic in $\Gamma_\ell(B)$.  Observe that the quotient complex $\Gamma_\ell'(G)/G$ is isomorphic to a wedge $X_\ell(A)\vee X_\ell(B)$, where $X_\ell(A)=\Gamma_\ell(A)/A$ and $X_\ell(B)=\Gamma_\ell(B)/B$. Then the closed path $\gamma\colon S^1\to X_\ell(A)\vee X_\ell(B)$  factors through the $\pi_1$-injective  inclusion $X_\ell(B) \hookrightarrow X_\ell(A)\vee X_\ell(B)$. It follows that there is a disk diagram $D\to X_\ell(B)$ for $\gamma\colon S^1 \to X_\ell(B)$ which can be lifted to a map $D\to \Gamma_\ell(B)$. Hence $\gamma$ is null-homotopic in $\Gamma_\ell(B)$. It follows that the inclusion 
\[\Gamma_\ell(B) \to \Gamma_\ell(G)\]
is a $\pi_1$-injective map. Thus, for each $\ell$, we have the following commutative diagram 
\[
\xymatrix{
\pi_1(\Gamma_\ell(B))\ar[d]\ar[r] & \pi_1(\Gamma_{\ell+1}(B))\ar[d]\\
\pi_1(\Gamma_\ell(G))\ar[r] & \pi_1(\Gamma_{\ell+1}(G))
}
\]
with injective vertical maps. The diagram shows that $H(B) \subseteq H(G)$. A completely analogous argument shows that $H(A) \subseteq H(G)$, and therefore $H(A)\cup H(B) \subseteq H(G)$.

%Suppose that $A$ has a finite presentation $\langle S_A | R_A \rangle$ and let $\ell_0 = \max\{ |r|\colon r\in R_A\}$. Let $\ell\geq \ell_0$. 

Next we prove $H(A)\cup H(B) \supseteq H(G)$. Let $\ell\in H(G)$. Let $\gamma$ be a closed path in $\Gamma_\ell(G)$ of length $\ell$ that is not null-homotopic. Suppose that the label of $\gamma$ is a word $w$. First let us argue that $\gamma$ is homogeneous, i.e, $w$ has only one syllable. Suppose that $w$ has more than one syllable. If $w$ represents the identity in $A\ast B$, then it contains a syllable that represents the identity. This gives a proper subpath of $\gamma$ that is a closed loop of length less than $\ell$. It follows that  $\gamma$ is homotopic to a closed path of shorter length in $\Gamma_{\ell}(G)$, and therefore $\gamma$ is null-homotopic, a contradiction. We have that $\gamma$ is homogeneous. Suppose  $\gamma$ has label a $B$-word of length $\ell$. Since   $\Gamma_\ell(B) \hookrightarrow \Gamma_\ell(G)$, by translating with an element of $G$, assume that $\gamma$ is a path in $\Gamma_\ell(B)$. Then  $\gamma$ is a closed path of length $\ell$ that is not null homotopic in $\Gamma_\ell(B)$ and hence $\ell\in H(B)$. Analogously we prove that if $\gamma$ has label an $A$-word, then $\ell\in H(A)$.
\end{proof}

\subsection{The taut spectrum of a small cancellation product}

Consider groups $A$ and $B$ and let   $\mathcal{R}$ be a subset of $A\ast_C B$ where $C$ is a common finite subgroup. We use the definitions, language and conventions by Lyndon and Schupp~\cite[Ch. V.]{LySc01}. In particular, a symmetrized subset   $\mathcal{R}$ of $A\ast_C B$ that satisfies the $C'(1/6)$ small cancellation is non-empty and every $r\in \mathcal{R}$ has normal form with at least seven alternating letters from  $A$ and $B$, see~\cite[Page 286]{LySc01}. We will use the following well known result in this section: 
 
\begin{theorem}\cite[Ch. V. Theorem 11.2]{LySc01}\label{thm:filling} 
Let $A$ and $B$ be groups, let $C$ be a common finite group, and let $R$ be a finite symmetrized subset of $(A \ast_C B)$ that satisfies the $C'(1/6)$ small cancellation condition. Let $G=(A\ast_C B)/\nclose{\mathcal R}$.
The natural homomorphisms $A \rightarrow G$ and $B \rightarrow G$ are injective, and we regard
    $A$ and $B$ as subgroups.
\end{theorem}

In the rest of the section we prove the following statement. We divert the proof to the end of the section.

\begin{theorem}\label{thm:taut:small:cancellation}Let $A$ and $B$ groups and let $\mathcal{R}$ be a finite set of elements of $A\ast B$ satisfying the $C'(1/6)$ small cancellation condition. Let $G=(A\ast B)/\langle\langle \mathcal R \rangle\rangle$, then 
$H(A)\cup H(B)$ and $H(G)$ are $k$-related for some $k$. 
\end{theorem}

The small cancellation conditions over free products are expressed  in terms of the length of normal forms. On the other hand, the definition of the taut spectrum is in terms of the word metric with respect to finite generating sets. \cref{lem:Dehn} below makes the connection between the two norms. It is essentially a translation of Theorem 9.3 in \cite[p.~293]{LySc01}.

Let $A$ and $B$ be groups with finite symmetric generating sets $X_A$ and $X_B$, and consider the generating set $X=X_A\sqcup X_B$ of $A\ast B$. We will be considering words over the alphabet $X$. For $g\in A\ast B$, denote by $|g|$ the length of its normal form. If $g\in A\ast B$ has normal form $g_1 g_2 \cdots g_n$ then each element $g_i$ is called a \emph{syllable} of $g$. If $f,g\in A\ast B$ the product $fg$ is reduced if the last syllable of $f$ and the first syllable of $g$ are in distinct factors. 

For a word $W$ over $X$,  let $|W|_X$ denote  the length of the word. A word $W$ over the alphabet $X$ is \emph{cyclically reduced} if when considered as a cyclic word, it is reduced. A word over the alphabet $X_A$ (resp. $X_B$) is geodesic if there is no  shorter word over $X_A$ (resp. $X_B$) that represents the same element of $A$ (resp. $B$). An $A$-syllable (resp. $B$-syllable) of a cyclically reduced  word is a maximal subword that uses only letters from $X_A$ (resp. $X_B$). An  $A$-syllable is geodesic if it is geodesic as a word over $X_A$. A geodesic $B$-syllable is defined analogously. A  word over $X$ is geodesic if all its syllables are geodesic.

If $U$ and $V$ are reduced words, we say that their concatenation $U\cdot V$ does not combine syllables if the last letter of $U$ and the first letter of $V$ are not both letters in $X_A$ or $X_B$. In particular, if $u,v \in A\ast B$ are the elements represented by $U$ and $V$, then the expression $uv$ is reduced.

Let $\calR$ be a nonempty symmetrized finite subset of $A\ast B$ that satisfies the $C'(\lambda)$ small cancellation condition with $\lambda\leq 1/6$, and let $N=\nclose{\calR}$ be the normal subgroup generated by $\calR$. Let 
\[M = \max\{ |W|_X \colon \text{$W$ is a syllable of a  geodesic word representing $r\in \calR$} \}.\]
Observe that $M\geq1$ is well defined since $X_A$, $X_B$ and $\calR\neq\emptyset$ are finite sets.   

\begin{lemma}[Sufficient condition for Dehn Algorithm]\label{lem:Dehn}
Suppose that $1\geq 3\lambda(M+1)$.
Let $W$ be a non-trivial reduced cyclic $X$-word that represents an element of  $N$. Suppose that each syllable of $W$ is a geodesic word. Then there is a cyclically reduced word $R$ representing an element of $\calR$ such that
\begin{enumerate}
    \item every syllable of $R$ is geodesic, 
    \item $W$ equals the concatenation of three words $U\cdot S\cdot V$ and the concatenation does not combine syllables,
    \item  $R$ equals the  concatenation $S\cdot T$, where $S$ is the word of the previous item,  and the concatenation does not combine syllables, and
    \item $|S|_X>|T|_X$.
\end{enumerate}
\end{lemma}
\begin{proof}  
%To differentiate between words and elements of the group, if $a$ is an element of the group then $a^*$ is a word that represent the element.  

Observe that $W$ is neither an $A$-word nor a $B$-word, since $\nclose{R}$ intersects trivially both $A$ and $B$.

Since $W$ represents a non-trivial element $w$ of $\nclose{R}$, then there is $r\in R$ such that $r=s t$ in reduced form, $w=usv$ in reduced form, and $|s|>(1-3\lambda)|r|$, see~\cite[Chapter 5, Theorem 9.3]{LySc01}. Since $|r|=|s|+|t|$,  the last inequality implies
\[ |s|> \frac{1-3\lambda}{3\lambda}|t|.\]
On the other hand,  $W$ is a concatenation of words  $W=U\cdot S\cdot V$ that does not combine syllables and $U, S, V$ represent the elements $u,s,v$ of $A\ast B$ respectively. Since $r=st$ is a reduced expression, then for any geodesic word $T$ representing $t$, 
the concatenation $R:=S\cdot T$ does not combine syllables and is a word that represents $r$. Since $1\geq 3\lambda(M+1)$, we have that $ 1-3\lambda  \geq 3\lambda M$, and therefore 
\[|S|_X \geq |s| > \frac{1-3\lambda}{3\lambda}|t| \geq \frac{1-3\lambda}{3\lambda M} |T|_X \geq |T|_X. \qedhere \]
\end{proof}

%The issue is that the small cancellation condition is in terms of the length of normal forms for words in $A\ast B$; on the other hand, the definition of the spectrum is in terms of the word metric with respect to $S=S_A\cup S_B$. Denote the length of normal forms $|w|$ and the length in terms of the finite generating set as $|w|_S$. 

From now on, we assume  $\mathcal{R}$  satisfies $C'(\lambda)$ with $\lambda\leq 1/6$ and moreover  that  $1\geq 3\lambda(M+1)$, as in the hypothesis of \cref{lem:Dehn}. Note that this second assumption can be obtained since $\mathcal R$ is finite and hence one can add all syllables of the elements of $\mathcal R$ to $X_A$ and $X_B$ which would mak $M=1$. 

Let $G=(A\ast B) / \nclose{\mathcal R}$, and let
\[ \ell_0 = M\max\{ |r|  \colon r\in \mathcal{R}\} .\]
Hence $\ell_0$ is an upper bound on  the maximal length of a geodesic word in $X_A\cup X_B$ representing an element of $\mathcal{R}$.

\begin{lemma}\label{lem:happy04}
Let $\ell> \ell_0$ and suppose $\ell \in H(G)$. Then there is $k\in H(A)\cup H(B)$ such that $\ell\leq k<2\ell$.
\end{lemma}
\begin{proof}
There is a closed path $\gamma$ of length $\ell$ that is not null-homotopic in $\Gamma_\ell(G)$. Observe that  $\gamma$ is an embedded path, and hence it is labeled by a cyclically reduced word $W$ such that each syllable represents a non-trivial element in the corresponding factor.

If $W$ is a $B$-word then, up to a translate, $\gamma$ lifts to a path in $\Gamma_\ell(B)$ that is not null-homotopic and hence $\ell\in H(B)$. We proceed analogously when $W$ is an $A$-word.

Suppose that $W$ has at least two syllables, i.e., $W$ is neither an $A$-word nor a $B$-word. 
 
Let us argue that there is a syllable of $W$ that is not geodesic. Suppose that all syllables of $W$ are geodesic. Then $W$ is a non-trivial cyclically reduced geodesic word, and Lemma~\ref{lem:Dehn} implies that there is a geodesic word $R=S\bar T$ representing an element of $\mathcal{R}$ such that  $W=U\cdot S\cdot V$ and $|S|>|T|$. Since $R$ is geodesic, it follows that $|R|\leq \ell_0<\ell$ and hence there is a path $\gamma'$ in $\Gamma_\ell(G)$ labeled by the word $U\cdot T\cdot V$ that is homotopic to $\gamma$. Since $|\gamma'|<|\gamma|=\ell$, it follows that $\gamma'$ is null homotopic in $\Gamma_\ell(G)$ and therefore $\gamma$ as well. This is a contradiction. Therefore not all syllables of $W$ are geodesic.

Suppose $S$ is a  $B$-syllable of $W$ that is not geodesic. Let $T$ be a geodesic $B$-word that represents the same element as $S$. Note that $S\cdot \bar T$ does not label a null homotopic path in $\Gamma_\ell(G)$ since the otherwise the $\ell$-path $\gamma$ would be homotopic to a shorter closed path in $\Gamma_\ell(G)$ and then $\gamma$ would be null homotopic which is not the case. Since $|T|<|S|<|W|=\ell$, it follows that the $B$-word $S\cdot \bar T$ has length strictly smaller than $2\ell-2$. Let $\beta$ be a path in $\Gamma_\ell(B)$ labeled by the $B$-word $S\cdot \bar T$. Since $\beta$ is not null homotopic in $\Gamma_\ell(B)$ and $|\beta|<2\ell$, it follows that there there is $\ell\leq k<2\ell$ such that $k\in H(B)$. We proceed analogously when $S$ is an $A$ syllable of $W$ that is not geodesic.
\end{proof}

For $\ell> 0$, define $\Gamma'_\ell(G)$ as the $2$-complex with $1$-skeleton the Cayley graph of $G$, and a $2$-cell for each homogeneous cycle of length less than $\ell$ and each closed path labeled by a geodesic word representing an  element of $\mathcal{R}$. Note that $G$ acts on  $\Gamma'_\ell(G)$. 

\begin{lemma}\label{lem:happy2}
Let $\ell>0$.
The map $\Gamma_\ell(B) \to \Gamma'_\ell(G)$ is $\pi_1$-injective.
\end{lemma}
\begin{proof}
Note that $\Gamma_\ell'(G)/G$ is obtained from the wedge of $\Gamma_\ell(A)/A$ and $\Gamma_\ell(B)/B$ by attaching 2-cells using the elements of $\mathcal R$. Thus \[\pi_1(\Gamma_\ell'(G)/G)=\left( \pi_1(\Gamma_\ell(A)/A)\ast \pi_1(\Gamma_\ell(B)/B) \right)/\langle\langle\mathcal R\rangle\rangle.\] Moreover, the relations still satisfy the $C'(1/6)$ small cancellation condition. Therefore $\Gamma_\ell(B)/B\to \Gamma_\ell'(G)/G$ is $\pi_1$-injective.
Now the proof is analogous to the one in the free product case.
\end{proof}

\begin{lemma}\label{lem:happy}
%Suppose that $2\ell\geq \ell_0$. 
Any closed path in $\Gamma_{2\ell}'(G)$ of length less than $\ell$ is null homotopic.
\end{lemma}
\begin{proof}
It is enough to prove the statement for embedded closed paths. Let $\alpha$ be an embedded closed path of length less than $\ell$ in $\Gamma_{2\ell}'(G)$. Let $W$ be the reduced cyclic word that labels the path $\alpha$. 

If $W$ is a monosyllable word then $\alpha$ is a   null-homotopic closed path in $\Gamma_{2\ell}'(G)$ by definition. 

If $W$ is not a monosyllable word and all syllables are geodesic then Lemma~\ref{lem:Dehn} implies that $\alpha$ is homotopic in $\Gamma_{2\ell}'(G)$ to a shorter closed path.

%Here we use that $2\ell\geq \ell_0$ and that any geodesic word representing an element of $\mathcal{R}$ has length at most $\ell_0$. 

If $W$ is not a monosyllable word and is not geodesic. Then any syllable has length less than $|W|<\ell$. Hence replacing each syllable by a geodesic syllable implies $\alpha$ is homotopic in $\Gamma_{2\ell}'(G)$ to a shorter closed path.   

The result of the lemma follows by induction on the length of the closed path.
\end{proof}

\begin{lemma}\label{lem:happy3}
%Suppose that $2\ell>\ell_0$.
If $\gamma\colon S^1\to \Gamma_{\ell}(G)$ is null homotopic, then $\gamma\colon S^1\to \Gamma_{2\ell}'(G)$ is null homotopic.
\end{lemma}
\begin{proof}
Since $\gamma\colon S^1\to \Gamma_{\ell}(G)$ is null homotopic, there is a cellular disk diagram $D\to \Gamma_{\ell}(G)$ whose boundary path is $\gamma$. Then the boundary path of any $2$-cell of $D$ maps to the boundary path of a $2$-cell of $\Gamma_{\ell}(G)$ and hence it has length less than $\ell$. By Lemma~\ref{lem:happy}, each $2$-cell of $D$ can be replaced by a disk diagram that maps into $\Gamma_{2\ell}'(G)$. Hence there is a disk diagram $D' \to \Gamma_{2\ell}'(G)$ whose boundary path is $\gamma$.
\end{proof}

\begin{lemma}\label{lem:happy05}
Let $\ell>\ell_0$ and suppose that $\ell\in H(A)\cup H(B)$. Then there is $k\in H(G)$ such that $\ell/2\leq k\leq \ell+1$. \end{lemma}
\begin{proof}
Let $\gamma$ be a closed path in $\Gamma_\ell(B)$ of length $\ell$ that is not null homotopic. Lemma~\ref{lem:happy2} implies that $\gamma$ is  not null homotopic in $\Gamma'_\ell(G)$. Then Lemma~\ref{lem:happy3} implies that $\gamma$ is not null homotopic in $\Gamma_{\ell/2}(G)$. Hence $\gamma$ is path of length $\ell$ that is not null homotopic in $\Gamma_{\ell/2}(G)$. It follows there is $k\in H(G)$ such that $\ell/2\leq k\leq \ell$.
We proceed in an analogous way when $\gamma$ be a closed path in $\Gamma_\ell(A)$ of length $\ell$ that is not null homotopic.
\end{proof}

\begin{proof}[Proof \Cref{thm:taut:small:cancellation}]
The statement of the theorem follows from Lemmas~\ref{lem:happy04} and~\ref{lem:happy05}.
\end{proof}

\section{Proper and virtually cyclic dimensions of multi-ended groups}\label{sec:dimensional:bounds}\label{sec:NewTheorem}

Let $G$ be a discrete group. Given a CW-complex $X$ with a cellular $G$-action, we say $X$ is a $G$-CW-complex if every element of $G$ that fixes setwise a cell, it fixes the cell pointwise. A non-empty collection $\calF$ of subgroups of
$G$ is called a \emph{family} if it
is closed under conjugation and under taking subgroups. 
A $G$-CW-complex $X$ is a \emph{model} for \emph{the classifying space} $E_{\calF}G$ if the following conditions are satisfied:
\begin{enumerate}
  \item For all $x\in X$, the isotropy group $G_x$ belongs to $\calF$.

  \item For all  $H\in\calF$ the subcomplex $X^H$  of $X$, consisting of points in $X$ that are fixed under all elements of $H$, is contractible. In particular $X^H$ is non-empty.
\end{enumerate}

The \emph{$\calF$-geometric dimension of $G$}, denoted $\gd_\calF(G)$, is the minimum $n$ for which there exists an $n$-dimensional model for $E_\calF G$, see \Cref{subsect:small:cancellation}. The  orbit category $\calO_\calF G$  has as objects the homogeneous $G$-spaces $G/H$, $H\in \calF$, and morphisms are  $G$-maps. A  \emph{$\calO_\calF G$-module} is a contraviariant functor from $\calO_\calF G$ to the category of abelian groups, and a morphism between two $\calO_\calF G$-modules is a natural transformation of the underlying functors. Denote by $\calO_\calF G$-mod the category of $\calO_\calF G$-modules. It turns out that $\calO_\calF G$-mod is an abelian category with enough projectives (see \cite[pg.~9]{MV03}). The
\emph{$\calF$-cohomological dimension of $G$}, denoted $\cd_{\calF}(G)$, is the length of the shortest projective resolution of the constant $\orf{G}$-module $\dbZ_\calF$, where $\dbZ_\calF$ is given by $\dbZ_\calF(G/H)=\dbZ$, for all $H\in \calF$, and every morphism of $\orf{G}$ goes to the identity function.
The following inequality is a standard result
\[\cd_\calF(G)\leq \gd_\calF(G), \]
see for instance \cite{LM00}.
We denote by $\cdfin(G)$ and $\cdvc(G)$ the cohomological dimensions for the families of finite and virtually cyclic subgroups respectively, and analogously for geometric dimensions. We also denote by $\underline{E}G$ and $\underline{\underline{E}}G$ the classifying spaces of $G$ for the families of finite and virtually cyclic subgroups respectively. In this section we prove the following results.

\begin{theorem}\label{cor:closed:under:takings:graphs}
Let $G$ be the fundamental group of a graph of groups with finite edge stabilizers, and let $T$ be the Bass-Serre tree.
Then the following inequalities hold
\[\gdfin(G)\leq \max\{1,\gdfin(G_{\sigma})| \ 
\sigma\in \ T \},\quad \cdfin(G)\leq \max\{1,\cdfin(G_{\sigma})| \ 
\sigma \in \ T \}\]
\[\gdvc(G)\leq \max\{2,\gdvc(G_{\sigma})| \ 
\sigma \in \ T \},\quad\cdvc(G)\leq \max\{2,\cdvc(G_{\sigma})| \ 
\sigma \in \ T \}.\]
\end{theorem}

\begin{theorem}\label{cor:megafinal-body}
%Let $G$ be the fundamental group of a graph of groups with finite edge stabilizers, and let $T$ be the Bass-Serre tree. 
Under the assumptions of \cref{cor:closed:under:takings:graphs}, if all vertex groups have small centralizers and satisfy the ascending chain condition for finite subgroups, then  
\[\gdvc(G)\leq \max\{2,\gdfin(G_{\sigma})| \ 
\sigma \in \ T \}\text{ and }\cdvc(G)\leq \max\{2,\cdfin(G_{\sigma})| \ 
\sigma \in \ T \}.\]
\end{theorem}

\subsection{An Auxiliary Result}
The following proposition is a classical result in homology theory, see for instance the Proof 1 of Proposition 1.1 in \cite{Br87}.

\begin{proposition}\label{prop:baby:haefliger}
Let $G$ be a group and let $R$ be a ring. Assume $G$ acts on an $R$-acyclic $G$-CW-complex $X$. Then 
\[\cd_{R}(G)\leq \max\{\cd_{R}(G_{\sigma})+\dim(\sigma)| \ 
\sigma \text{ is a cell of } \ X \}.\] 
\end{proposition}

In this section we prove the following generalization of the above inequality for cohomological and geometric dimension relative to families of subgroups. The proof of the cohomological part of this is adapted from an argument by Fluch and Nucinkis~\cite[Proof of Proposition 4.1]{FN13}.
The geometric part is implicit in \cite[Proof of Theorem~3.1]{Lu00}. 

\begin{proposition}\label{coro:haefliger}
Let $G$ be a group. Let $\calF$ and $\calG$ be families of subgroups of $G$ such that $\calF\subseteq \calG$. If $X$ is a model for $E_\calG G$, then
\[\gd_{\calF}(G)\leq \max\{\gd_{\calF\cap G_\sigma}(G_{\sigma})+\dim(\sigma)| \ 
\sigma \text{ is a cell of } \ X \}\]
and 
\[\cd_{\calF}(G)\leq \max\{\cd_{\calF\cap G_\sigma}(G_{\sigma})+\dim(\sigma)| \ 
\sigma \text{ is a cell of } \ X \}.\]
\end{proposition}

The next theorem taken from \cite[Theorem~2.3]{MPSS20} will be used to prove \cref{coro:haefliger}. Given a group $G$, a family of subgroups $\calG$ and a 
$G$-CW-complex,
 we say that $X$ is a $\calG$-G-CW-complex if all isotropy groups of $X$ belong to $\calG$. 

\begin{theorem}[The Haefliger-L\"uck construction for families]\label{thm:haefliger}
Let $G$ be a group. Let $\calF$ and $\calG$ be families of subgroups of $G$ such that $\calF\subseteq \calG$. Consider a $\calG$-$G$-CW-complex $X$. For each cell $\sigma$ of $X$, fix models $X_\sigma$ for $E_{\calF\cap G_\sigma} G_\sigma$ so that for any two cells in the same $G$-orbit the same model has been chosen. Then, for each $n\geq 0$ there exists an $\calF$-$G$-CW-complex $\hat X_n$ and a $G$-map $f_n\colon \hat X_n\to X^{(n)}$ such that:
\begin{enumerate}
    \item We have $\hat X_{n-1}\subseteq \hat X_n$ and $f_n$ restricted to $\hat X_{n-1}$ is $f_{n-1}$.
    \item For every open simplex $\sigma$ of $X^{(n)}$, $f_n^{-1}(\sigma)$ is a model for $E_{\calF\cap G_\sigma} G_\sigma$.
    \item For all $H\in \calF$, $f_n^H\colon \hat X_n^H\to (X^{(n)})^H$ is a (nonequivariant) homotopy equivalence.
    \item There is a $G$-equivariant homeomorphism between $f_n^{-1}(\sigma)$ and $X_\sigma \times \sigma$.
    \end{enumerate}
\end{theorem}

\begin{proof}[Proof of \Cref{coro:haefliger}] 
For each $n\geq 0$ let $\hat X_n$ be the spaces provided by \cref{thm:haefliger} with $X_\sigma$ models of minimal dimension for each cell $\sigma$ of $X$, i.e. $\dim(X_\sigma)=\gd_{\calF\cap G_\sigma}(G_\sigma)$. Denote $\hat X=\bigcup_{n=0}^\infty \hat{X}_n$. As consequence of hypothesis (3) of \cref{thm:haefliger} we have that  $\hat X^K$ is contractible for every $K\in \calF$, therefore $\hat X$ is a model for $E_\calF G$. By hypothesis (4) of \cref{thm:haefliger} we have that
\[\dim(\hat X)=\max\{\gd_{\calF\cap G_\sigma}(G_{\sigma})+\dim(\sigma)| \ 
\sigma \text{ is a cell of } \ X \}\]
and the statement for $\calF$-geometric dimension follows.

Now we prove the statement for the $\calF$-cohomological dimension. Denote by $\mathrm{res}^\calG_\calF\colon \calO_\calG G\text{-mod}\to \calO_\calF G\text{-mod}$ the natural restriction functor. For $P$  a $\calO_\calF$-module, denote by $\mathrm{pd}_{\calF}(P)$ the projective dimension of $P$, i.e. the smallest number $n$ such that $P$ admits a $\calO_\calF G$-projective resolution of length $n$.

The augmented Bredon cellular complex of $\hat X$
\[\cdots\to \bigoplus_{\text{2-cells}} \dbZ[-,G/G_{\sigma^2}]\to \bigoplus_{\text{1-cells}}\dbZ[-,G/G_{\sigma^1}] \to \bigoplus_{\text{0-cells}}\dbZ[-,G/G_{\sigma^0}]\to \dbZ_\calG\to 0\]
is a free resolution of $\dbZ_\calG$ of length $\dim(\hat X)$. Following the proof of Proposition 4.1 in \cite{FN13} we have that
\[\cdots\to  \bigoplus_{\text{1-cells}}\mathrm{res}_\calF^\calG \dbZ[-,G/G_{\sigma^1}] \to  \bigoplus_{\text{0-cells}}\mathrm{res}_\calF^\calG\dbZ[-,G/G_{\sigma^0}]\to \mathrm{res}_\calF^\calG(\dbZ_\calG)=\dbZ_\calF\to 0\]
is a resolution for $\dbZ_\calF$ in $\calO_\calF G$-mod, and  $\mathrm{pd}_\calF(\mathrm{res}_\calF^\calG \dbZ[-,G/G_{\sigma}])=\cd_{\calF\cap G_{\sigma}}(G_\sigma)$  for each cell $\sigma$  of $\hat X$. Now, a standard argument let us assemble all the projective resolutions of the $\mathrm{res}_\calF^\calG \dbZ[-,G/G_{\sigma}]$ to obtain to obtain a projective resolution of $\dbZ_\calF$ of dimension 
\[\max\{\cd_{\calF\cap G_\sigma}(G_{\sigma})+\dim(\sigma)| \ 
\sigma \text{ is a cell of } \ X \}.\qedhere\]
\end{proof}

\subsection{Proof of Theorem~\ref{cor:closed:under:takings:graphs}}

\begin{proof}%[Proof of Theorem~\ref{cor:closed:under:takings:graphs}]
If $G$ is the fundamental group of a graph of groups with  finite edge groups, then $G$ acts on its Bass-Serre tree $T$ with finite edge stabilizers. For every subgroup $H$ of $G$, the fixed point set $T^H$ is either empty or it is convex (and therefore contractible), thus $T$ is a classifying space for $G$ with respect to the smallest family $\calG$ of $G$ that contains all vertex subgroups. It is well-known that every finite group acting on a tree fixes at least one point, thus the family $\calF$ of finite subgroups of $G$ is contained in $\calG$.  By \cref{coro:haefliger} and the fact that $\gdfin(F)=0$ for every finite group $F$, we obtain $\gdfin(G)\leq \max\{1,\gdfin(G_{\sigma})| \ 
\sigma \text{ is a vertex of } \ X \}$. This proves the first claim.

For the second claim note that the action of $G$ on $T$ is acylindrical, that is, the stabilizer of every path of length at least 1 is finite. By \cite[Theorem~6.3]{LASS21} there exists a family $\calH$ of subgroups of $G$ and a 2-dimensional model $X$ for $E_{\calH}G$ such that  $\vcyc\subseteq \calH$. Such a a model $X$ is obtained from $T$ after coning-off all geodesics that admit a cocompact action of a virtually cyclic subgroup of $G$, and the stabilizer of every cone point is virtually cyclic as a consequence of the acylindricity of the action of $G$ on $T$. Now we proceed to analyze the dimension of stabilizers of cells of $X$. A 0-cell it is either a vertex of $T$ or it is a cone point. Each cone point has virtually cyclic stabilizer, thus it has virtually cyclic dimension 0. A 1-cell it is either an edge of $T$ or connects a vertex of $T$ with a cone point. In the first case the stabilizer is finite and has virtually cyclic dimension 0, while in the second case the stabilizer of the edge also stabilizes the a cone point, thus it also has virtually cyclic dimension 0. Every 2-cell of $X$ contains in its boundary a cone point, thus its stabilizer also has virtually cyclic dimension 0.
Now we can apply \cref{coro:haefliger} to $X$ and $\vcyc\subseteq \calH$ to obtain $\gdvc(G)\leq \max\{2,\gdvc(G_{\sigma})| \ 
\sigma \text{ is a vertex of } \ T \}$.

For the inequalities for cohomological dimension we run the same proof using the cohomological conclusions of \cref{coro:haefliger}.
\end{proof}

\subsection{Proof of \cref{cor:megafinal-body}}

The result is obtained by putting together the inequalities for $\gdvc(G)$ and $\cdvc{G}$ given by \Cref{cor:closed:under:takings:graphs} and \Cref{lemma:bound:dimensions} below. Let us remark that the geometric part of \Cref{lemma:bound:dimensions} is based on an argument by Juan-Pineda and Leary~\cite{JPL06}  for the particular case that $G$ is a hyperbolic group. The cohomological part of \Cref{lemma:bound:dimensions} takes ideas from Fluch and Leary in~\cite{FL14}.

\begin{proposition}\label{lemma:bound:dimensions}
Let $G$ be a group that has small centralizers, and satisfies the ascending chain condition for finite subgroups. Then
\[\gdvc(G)\leq \max\{\gdfin(G),2\},\]
and 
\[\cdvc(G)\leq \max\{\cdfin(G),2\}.\]
\end{proposition}
\begin{proof}
By our hypothesis on $G$ and \cite[Theorem~3.1]{LW12} we have that every infinite virtually cyclic subgroup of $G$ is contained in a unique maximal virtually cyclic subgroup of $G$, and $N_G(V)=V$ for every infinite maximal virtually cyclic subgroup of $G$. Thus by \cite[Corollary~2.11]{LW12} we obtain a model for $\evc G$ as the $G$-CW-complex $Y$ defined by the following homotopy $G$-pushout
\begin{equation*}
\begin{tikzcd}
\displaystyle\bigsqcup_{V\in \mathcal{M}}G\times_{V}\underline{E}V \arrow[r, "g"] \arrow[d, "f"]
& \underline{E}G  \arrow[d ] \\
\displaystyle\bigsqcup_{V\in \mathcal{M}}G/V \arrow[r ]
& Y
\end{tikzcd}
\end{equation*}
where $\mathcal M$ is a complete system of representatives of the conjugacy classes of maximal infinite virtually cyclic subgroups of $G$, $g$ is the inclusion map, and $f$ is the disjoint union of the projection maps $G\times_V\underline{E}V\to G/V$. Hence if we take each $\underline{E}V$ to be the real line, and $\underline{E}G$ a space of dimension $\gdfin(G)$, then $Y$ is obtained from $\underline{E}G$ by coning-off certain lines (the images of the maps $\underline{E}V\to \underline{E}G$). Since $Y$ is already a model for $\evc G$ of dimension $\max\{\gdfin(G),2\}$  the first claim follows.

For the second claim, we can run verbatim the argument in the proof of Theorem 7 in \cite{FL14}. We include this argument for the sake of completeness. Let $\calF$ and $\calG$ be the families of finite subgroups and virtually cyclic subgroups of $G$ respectively. Consider the following short exact sequence of $\calO_\calG G$-modules
\[1\to \bar\dbZ_{\calF}\to \dbZ_{\calG} \to Q\to 1\]
where $Q$ is the $\calO_\calG G$-module given by $Q(G/H)=\dbZ$ if $H\not\in \calF$ and $Q(G/H)=0$ if $H\in \calF$, and $\bar\dbZ_{\calF}$ is the $\calO_\calG G$-module given by $\bar\dbZ_\calF(G/H)=\dbZ$ if $H\in \calF$ and $\bar\dbZ_\calF(G/H)=0$ if $H\not\in \calF$. Since  $Y$ is a model for $\underline{\underline{E}}G$ obtained from a model $X$ for $\underline{E}G$ by attaching cells of dimension at most $2$, the Bredon cellular chain complex $C_*(Y,X)$ yields a  projective resolution for $Q$ of length $2$. Let $\mathbf{P}_*\to \dbZ_\calF$ be a projective resolution of length $\cdfin(G)$, and promote it to a projective resolution $\bar{\mathbf{P}}_*$ for $\bar\dbZ_\calF$ with the same length by setting $\bar{\mathbf{P}}_*(G/H)=0$ for every $H\not\in \calF$.  We can apply the Horseshoe lemma to the resolutions $C_*(Y,X)$ and $\bar{\mathbf{P}}_*$ to  obtain a projective resolution for $\dbZ_{\calG}$ of length $\max\{\cdfin(G),2\}$.  This concludes the proof.
\end{proof}

\section{Proper and virtually cyclic dimensions of small cancellation quotients of free products}\label{sec:Thn1.5}

This section contains the proof of the following statement.

\begin{theorem}\label{pre-Machine-body} 
Let $A$ and $B$ be groups and let $C$ be a common finite subgroup.  Let $\mathcal R$ be a finite symmetrized subset of $A\ast_C B$ that satisfies the $C'(1/12)$ small cancellation condition, and $G = (A\ast_C B)/ \nclose{\mathcal R}$. Assume $G$ is not virtually free. Then the following conclusions hold.
\begin{enumerate}
    \item $ \gdfin(G) =\max\{ \gdfin(A),\gdfin(B), 2\} \text{ and }   \cdfin(G) =\max\{ \cdfin(A),\cdfin(B), 2\}$
    
     \item If  $A$ and $B$ are finitely generated, have small centralizers and satisfy the  ascending chain condition for finite subgroups, then $G$ has small centralizers, satisfies the ascending chain
    condition, and
     \[\max\{\gdvc(A),\gdvc(B),2\}\leq \gdvc(G)\leq\max\{\gdfin(A),\gdfin(B),2\}.\]
     and
    \[\max\{\cdvc(A),\cdvc(B),2\}\leq \cdvc(G)\leq\max\{\cdfin(A),\cdfin(B),2\}.\]
    
    \item For any ring $R$, we have $ \cd_{R}(G) =\max\{ \cd_{R}(A),\cd_{R}(B), 2\} .$ 
\end{enumerate}
\end{theorem}

\subsection{A Cocompact  Model for Small Cancellation Quotients}\label{subsect:small:cancellation}

For a small cancellation product $G=(A\ast B)/\nclose{\mathcal R}$ there is a natural family of subgroups to consider, namely, the smallest family $\calF$ that contains $\{A,B\}$. In this section we describe a canonical  2-dimensional model for $E_\calF G$.
 
\begin{definition}[Coned-off Cayley complex $\hat X$]
Let $A$ and $B$ be  groups, let $C$ be a common finite subgroup, and let $\mathcal R$ be a finite  subset of $A \ast_C B$. Let $G=(A \ast_C B)/\nclose{\mathcal R}$. Suppose that 
 the natural homomorphisms $A \rightarrow G$ and $B \rightarrow G$ are injective, and  regard
    $A$ and $B$ as subgroups of $G$.
    
    Let \emph{$\hat \Gamma$  } be the $G$-graph with vertex set the $G$-set of left cosets of $A$ and $B$ in $G$, and edge set   $\left\{\{gA,gB\}\colon g\in  G \right\} $. Equivalently, if   $T$ is the Bass-Serre tree of $A\ast_C B$, then $\hat\Gamma = T/\nclose{\mathcal R}$. Note that $\hat \Gamma$ is a cocompact connected $G$-graph with finite edge stabilizers. Moreover, the stabilizer of every vertex is either a  a conjugate of $A$ or a conjugate of $B$. 
    
    The \emph{coned-off Cayley complex $\hat X$} of $G$ is a $2$-dimensional $G$-complex  with $1$-skeleton $\hat\Gamma$ defined as follows. Since  the natural morphisms $A\to G$ and $B \to G$ are injective,  the subgroup $\nclose{\mathcal R}$  intersects trivially with  $A$ and $B$. Thus the action of $N=\nclose{\mathcal R}$ on the Bass-Serre tree $T$ is free, and the quotient map $\rho\colon T\to T/N$ is a covering map. Fix a vertex $x_0$  of $T$ and consider it as a base point. Then any element $g$ of $ \nclose{R}$ induces a unique path $\alpha_g$ from $x_0$ to $gx_0$. Let $\gamma_g=\rho\circ \alpha_g$ be the closed path in $\hat \Gamma$ induced by $\alpha_g$ based at $\rho(x_0)$. This induces an isomorphism  $N\to \pi_1(\hat\Gamma, \rho(x_0))$ defined by $g\mapsto \gamma_g$.  For $g\in G$ and $h\in N$,  let $g\cdot\gamma_h$ be the translated closed path in $\hat \Gamma$ without an initial point, i.e., these are cellular maps from $S^1 \to \hat\Gamma$. Consider the $G$-set $\Omega=\{g.\gamma_{r}\mid r\in \mathcal R, g\in G\}$ of closed paths in $\hat\Gamma$. The complex $\hat X$  is then  obtained by attaching a $2$-cell to $\hat\Gamma$ for every closed path in $\Omega$. In particular, no pair of distinct $2$-cells of $\hat X$ have the same boundary path, and the pointwise $G$-stabilizer  of a $2$-cell of $\hat X$ coincides with the pointwise $G$-stabilizer of its boundary path (in particular they are finite subgroups). The natural isomorphism from $N$ to $\pi_1(\hat X^{(1)}, \rho(x_0))$ implies that $\hat X$ is simply connected. Moreover,  the $G$-action is cocompact since $\mathcal R$ is finite.
\end{definition} 

 \begin{proposition}\label{prop:ConedOffCayleyComplex}
 Let $A$ and $B$ be groups, let $C$ be a common finite subgroup, and let $R$ be a finite symmetrized subset of $A \ast_C B$ that satisfies the $C'(1/12)$ small cancellation condition. Let $G=(A\ast_C B)/\nclose{\mathcal R}$.  If $\mathcal{F}$ is the family of subgroups generated by the $G$-stabilizers of cells of $\hat X$, then $\hat X$ is a 2-dimensional cocompact model for  $E_\mathcal{F} G$ such that 
 the stabilizers of 1 and 2-cells are finite, and the vertex stabilizers are conjugates of $A$ or $B$. 
\end{proposition}

 Recall that a collection of subgroups $\calP$ of a group $G$ is \emph{almost malnormal} if for any $P,P'\in\calP$ and $g\in G$, either $gPg^{-1} \cap P'$ is finite, or $P=P'$ and $g\in P$. 

\begin{proof}[Proof of \cref{prop:ConedOffCayleyComplex}]
The argument that if $R$ satisfies the $C'(1/12)$ small cancellation condition, then  $\hat X$ is a contractible  $C'(1/6)$ small cancellation complex can be found in~\cite[Proof of Theorem 7.1]{ACCMP} in a slightly more general context; see \Cref{rem:SmallCancellation} for an additional comment. 

Before showing that $\hat X$ is a model of  $E_{\mathcal{F}}G$ let us make a remark that small cancellation products are standard examples of relatively hyperbolic groups~\cite[Page 4]{osin}. In particular, since $G$ is hyperbolic relative to the collection of subgroups $\{A,B\}$,   \cite[Proposition 2.36]{osin} implies that  $\{A,B\}$  is an almost malnormal collection in $G$. 

Let $K$ be a subgroup in $\mathcal{F}$. If $K$ is a finite subgroup, then $\hat X^K$ is a contractible subcomplex by~\cite[Proposition 5.7]{HaMa13}. Suppose  $K$ is infinite. Then the fixed point set $\hat X^K$ consists only of vertices, since $1$-cells and $2$-cells have finite stabilizers. Since  $\{A,B\}$ is an almost malnormal collection, the $G$-stabilizers of any two distinct vertices of $\hat X$ have finite intersection. It follows that $\hat X^K$ consists of a single vertex and hence it is a contractible subcomplex.
\end{proof}

\begin{remark}[Alternative argument proving  \Cref{prop:ConedOffCayleyComplex}]
Let $\hat\Gamma$ be the one-skeleton of $\hat X$ and note that this  is the coned-off Cayley graph of $G$ with respect to $\{A,B\}$.  
Since $G$ is hyperbolic relative to $\{A,B\}$, we have that $\hat \Gamma$ is a fine graph in the sense of Bowditch~\cite{Bo12}, see for example~\cite[Proposition 4.3]{MPW11b}. Then it is a direct  consequence  of~\cite[Corollary 1.6]{AM21} that the coned-off Cayley complex $\hat X$ is a  cocompact model for  $E_\mathcal{F} G$.  
\end{remark}

\begin{remark}[On the $C'(1/12)$ hypothesis of \Cref{prop:ConedOffCayleyComplex}] \label{rem:SmallCancellation}
A notable hypothesis in the statement of \Cref{prop:ConedOffCayleyComplex} is the $C'(1/12)$ small cancellation condition. This comes from the fact that    $(A\ast_C B)/\nclose{\mathcal R}$ being a   $C'(\lambda)$-small cancellation quotient does not imply that the coned-off Cayley complex $\hat X$ is a $C'(\lambda)$ small cancellation complex. This phenomenon regarding a difference between  algebraic and geometric versions of small cancellation conditions in the case of quotients of free products have been observed in the literature. For example, in~\cite[Page 2404]{MPW11}, this issue is illustrated
by considering the free product  $\mathbb{Z}_2 \ast \mathbb{Z}_3$ and the relation $(ab)^7$ where $a$ and $b$ are generators of $\mathbb{Z}_2$ and $\mathbb{Z}_3$ respectively. Then $(\mathbb{Z}_2 \ast \mathbb{Z}_3)/\nclose{(ab)^7}$ is a $C'(1/7)$ small cancellation quotient of the free product $(\mathbb{Z}_2 \ast \mathbb{Z}_3)$,  the coned-off Cayley complex $\hat X$ is $C'(1/7+\epsilon)$ for any $\epsilon>0$, but it is not $C'(1/7)$. In this particular case, $\hat X$ is a subdivision of the $\{7,3\}$ tiling of the hyperbolic plane.    
\end{remark}

\subsection{Proof of \cref{pre-Machine-body}}
The proof relies on the following lemma.

\begin{lemma}\label{lemma:small:centralizers}
Let $A$ and $B$ be groups and let $C$ be a common finite subgroup.  Let $\mathcal R$ be a finite symmetrized subset of $A\ast_C B$ that satisfies the $C'(1/12)$ small cancellation condition, and $G = (A\ast_C B)/ \nclose{\mathcal R}$. If $A$ and $B$ have small centralizers, then $G$ has small centralizers.
\end{lemma}
\begin{proof}
Recall that by \cref{thm:filling} we can regard $A$ and $B$ as subgroups of $G$. Let $g\in A$ be an element of infinite order, and let $x\in C_G(\langle g\rangle)$. Since $G$ is hyperbolic relative to $\{A,B\}$, we have that $\{A,B\}$ is an almost malnormal collection. Since $A$ is infinite,   $g\in xAx^{-1}\cap A$ implies $x\in A$. Therefore  $C_G(\langle g\rangle)=C_A(\langle g\rangle)$ and hence $\langle g\rangle$ has finite index in $C_G(\langle g\rangle)$. For $g\in B$ we proceed analogously. Finally if $g$ is neither subconjugate to  $A$ nor $B$, then by \cite[Theorem~1.14]{osin} we have that $\langle g\rangle$ has finite index in its centralizer.
\end{proof}

\begin{proof}[Proof of \cref{pre-Machine-body}]
\begin{enumerate}
    \item Consider the family $\mathcal{F}$ of subgroups generated by the $G$-stabilizers of cells of the coned-off Cayley complex $\hat X$.  \cref{prop:ConedOffCayleyComplex} implies that $\hat X$  is a 2-dimensional model of $E_{\calF}G$ such that the stabilizers of 1 and 2-cells are finite, and the vertex stabilizers are conjugates of $A$ or $B$. By \cite[Proposition~5.7]{HaMa13} every finite subgroup of $G$ fixes a point of $\hat X$, therefore  $\fin \subseteq \calF$ where $\fin$ is the family of finite subgroups of $G$. By \cref{coro:haefliger} applied for the case $\calF=\calG$, our  hypothesis on the dimensions of $A$ and $B$ and the fact that the stabilizers of all 1-cells and 2-cells of $\hat X$ are finite, we have 
    \begin{align*}
        \gdfin(G)&\leq \max\{\gdfin(G_{\sigma})+\dim(\sigma)| \ 
\sigma \text{ is a cell of } \ \hat X \}\\&\leq \max\{\gdfin(A),\gdfin(B),2 \}
    \end{align*}
On the other hand, since $A\leq G$, we have $\gdfin(G) \geq \gdfin(A)$. A completely analogous argument using the cohomological dimension part of \cref{coro:haefliger}  let us conclude that $\cdfin(G)\leq \max\{\gdfin(A),\gdfin(B),2 \}$. Since $G$ is not finite nor virtually free, we have that both $\gdfin(G)$  and $\cdfin(G)$ are greater than or equal to 2. This concludes the proof.

\item If $A$ and $B$ satisfy the ascending chain condition for finite subgroups, then $G$ also satisfies the  ascending chain condition for finite subgroups. Indeed, since $G$ is a small cancellation product of $A$ and $B$, every finite subgroup of $G$ is subconjugate to $A$ or $B$ or is a cyclic subgroup arising from a proper  power in $\mathcal R$. Since $\mathcal R$ is finite, the ascending chain condition for finite subgroups holds for $G$.
By our hypothesis and part (1) of this Theorem, we have $ \gdfin(G) =\max\{ \gdfin(A),\gdfin(B), 2\} \text{ and }   \cdfin(G) =\max\{ \cdfin(A),\cdfin(B), 2\}$. By \cref{lemma:small:centralizers}, $G$ has small centralizers. Thus by \cref{lemma:bound:dimensions} we have 
\begin{align*}
    \gdvc(G)&\leq \max\{\gdfin(G),2\}\\
    & \leq \max\{\gdfin(A),\gdfin(B),2\}
\end{align*}
Since $A$ and $B$ are finitely generated, then we can conclude $\gdvc(G)\geq 2$, see for instance \cite[Lemma~2.2]{LASS21}. Since $\gdvc(A)\leq \gdvc(G)$ and $\gdvc(B)\leq \gdvc(G)$ we conclude \[\max\{\gdvc(A),\gdvc(B),2\}\leq \gdvc(G)\leq\max\{\gdfin(A),\gdfin(B),2\}.\]
The analogous statement for cohomological dimension follows in exactly the same way.
\item To obtain that $ \cd_{R}(G) =\max\{ \cd_{R}(A),\cd_{R}(B), 2\}$ for any ring $R$, we run the same argument used in (1) considering the action of $G$ on $\hat X$ and using \cref{prop:baby:haefliger}.
\end{enumerate}

\end{proof}

\section{One-ended small cancellation quotients of free products}
\label{sec:SmallCancellation}

\begin{theorem} \label{prop:OneEndedProducts}
Let $G=(A\ast_C B)/\langle\langle \mathcal R \rangle\rangle$, where $C$ is a common finite subgroup of $A$ and $B$, and $\mathcal R$ is a symmetrized subset of $A\ast_C B$ that satisfies the $C'(1/6)$ small cancellation condition.
\begin{itemize}
\item Suppose that for  every $r\in \mathcal{R}$, its normal form does not  contain   elements of  finite order of $A$ or $B$.
\end{itemize}
If $A$ and $B$ are one-ended, then $G$ is one-ended.
\end{theorem}

We are not aware of an example that shows that  the bulleted hypothesis of Proposition~\ref{prop:OneEndedProducts} is necessary. The proof of \Cref{prop:OneEndedProducts} relies on the following lemma:

\begin{lemma}[Ping-pong for elliptic elements acting on trees]\label{lem:PingPong}
Let $A$ and $B$ be subgroups of the $\Aut(T)$ where $T$ is a simplicial tree. Suppose
\begin{enumerate}
  % \item $A\cap B$ is the trivial subgroup.
   \item $A$ and $B$ act with finite edge stabilizers on $T$.
    \item $A$ has a global fixed point $v_A$,  $B$ has a global fixed point $v_B$, and $v_A\neq v_B$.
\end{enumerate}
Suppose that $C$ is  a finite common subgroup of $A$ and $B$. Consider the natural map $\varphi\colon A\ast_C B \to \Aut(T)$.
If $r\in \ker(\varphi)$, then either $r$ is the identity element, or its normal form contains elements of finite order of $A$ or $B$.  In particular, if $A$ and $B$ are torsion-free, then $A\ast B \to \Aut(T)$ is injective.
\end{lemma}
\begin{proof}
Consider $T$ as metric space with the path metric arising by regarding each edge as an interval of length one. Then the groups $A$ and $B$ act faithfully by isometries on $T$. For an arbitrary point $v$ of $T$, we say that $v$ points to $A$ (resp. $B$) if the geodesic from $v$ to $v_B$ (resp. $v_A$) contains $v_A$ (resp. $v_B$). In particular, if $v$ points to $A$ (resp. $B$) then $\dist(v,\{v_A,v_B\})$ equals $\dist(v,v_A)$ (resp. $\dist(v,v_B)$).

Observe that if $v$ points to $A$ and $b\in B$ has infinite order, then $b.v$ points to $B$ and 
\[\dist(v,\{v_A, v_B\}) <  \dist(b.v,\{v_A, v_B\}).\]
This is a direct consequence of $T$ being a tree and that the fixed point set of $b$ is the single vertex $v_B$ (since $b$ has infinite order it can not fix an edge of the tree). Analogously, if $v$ points to $B$ and $a\in A$ has infinite order then $a.v$ points to $A$ and  
\[\dist(v,\{v_A, v_B\}) <  \dist(a.v,\{v_A, v_B\}).\]

Let $r\in A\ast_C B$ be a non-trivial element such that its normal form $r=g_1g_2\cdots g_{k}$ only involves elements of infinite order of $A$ and $B$. Let $v_0$ be the midpoint of  the geodesic path from $v_A$ to $v_B$. Note that $v_0$ neither points to $A$ nor $B$. Note that if $g_k\in A$ then $g_k.v_0$ points to $A$; and if $g_k\in B$ then $g_k.v_0$ points to $B$. By the statement of the previous paragraph,  $r.v_0$ points either to $A$ or $B$, and hence $r.v_0 \neq v_0$. Hence $r\not\in \ker(\varphi)$.
\end{proof}

\begin{proof}[Proof of \Cref{prop:OneEndedProducts}]
Suppose that $G$ is multiended. By Stallings, $G$ admits an non-trivial action $G\to \Aut(T)$ on a simplicial tree  $T$ with finite edge stabilizers and without inversions. Since $A$ is one-ended and edge stabilizers are finite, it follows that the fixed point set of $A$ in $T$ is a single vertex $v_A$, and by symmetry $B$ fixes a single vertex $v_B$. Since $A$ and $B$ generate $G$, and the $G$-action on $T$ is non-trivial, we have that $v_A \neq v_B$. 
Restrictions of the $G$-action   induce  maps $A\to \Aut(T)$ and $B\to \Aut(T)$; and these two morphisms induced a morphism $A\ast_C B \to \Aut(T)$ such that
\[\begin{tikzcd}
A\ast_C B \arrow[r] \arrow[d] & \Aut(T)\\ 
G   \arrow[ru] & \\ \end{tikzcd}\]
is a commutative diagram. Let $r\in\mathcal R$. By definition of $C'(1/6)$ small cancellation condition $r$ is non-trivial, hence by  Lemma~\ref{lem:PingPong}, $r$ contains elements of finite order in its normal form, which is a contradiction. 
\end{proof}

The proof of the following proposition was communicated to us by Dani Wise.

\begin{proposition}\label{prop:wise2}
Let $G$ be a torsion-free, $2$-generated group. If $G$ is not a free group, then $G$ is one-ended.
\end{proposition}
\begin{proof}
Suppose $G$ is multiended. Then $G$ splits as a free product or an HNN extension over a trivial group. First, suppose $G$ is a free product $A\ast B$. Since $G$ is $2$-generated,  Grushko's theorem implies that $A$ and $B$ are cyclic. Then $G$ being torsion-free implies that $A$ and $B$ are infinite cyclic groups. It follows  that  $G$ is a free group of rank two.  Suppose that $G$ is an  HNN extension with a trivial edge group. Then $G=H\ast \dbZ$ for some subgroup $H$ of $G$. If $H$ is non-trivial, we are in the previous case, and $G$ is a free group of rank two. If $H$ is trivial then $G$ is a free group of rank one. 
\end{proof}

%\begin{proposition}\label{prop:wise}Let $G=\langle a,b | \mathcal R\rangle$ where $\mathcal R$ is a symmetrized subset of $\langle a,b\rangle$ that satisfies the $C'(1/6)$ small cancellation condition. If $G$ is torsion free, then $G$ is one-ended.\end{proposition}
%\begin{proof}Suppose $G$ is $2$-generated,  torsion-free and multiended. Then $G$ splits as a free product or an HNN extension over a trivial group. If $G$ is a $2$-generated free product $A\ast B$, then by Grushko theorem $A$ and $B$ must be isomorphic to infinite cyclic groups, that is, $G$ is isomorphic to a free group of rank two.  It follows that the   natural surjection $\langle a,b \rangle \to G$ must be an isomorphism since free groups are hopfian. This implies that every element of $\mathcal R$ is trivial, which is a contradiction. If $G$ is an  HNN extension with a trivial edge group, then $G=H\ast \dbZ$ for some subgroup $H$ of $G$. If $H$ is non-trivial, we are in the above case again and we are done. \textcolor{red}{Why can we rule out the $H$ trivial case? I think we have to add the hypothesis that $G$ is non-cyclic, or equivalently that it is generated by exactly two elements.}\end{proof}

%\begin{proposition} Let $\mathcal{C}$ be one the classes of finitely generated groups in Theorem~\ref{thm:main}. If $G\in \mathcal{C}$ and $G$ splits as $A\ast_C B $ or $A\ast_C$ over a finite group $C$, then in the first case $A\in \mathcal{C}$ or $B\in \mathcal{C}$, and in the second case $A\in \mathcal{C}$. \end{proposition}

\bibliographystyle{alpha} %harvard, unsrt, alpha
\bibliography{myblib}

\begin{thebibliography}{ACCCMP0}

\bibitem[ACCCMP0]{ACCMP}
S.~Arora, I.~Castellano, G.~Corob~Cook, and E.~Martínez-Pedroza.
\newblock Subgroups, hyperbolicity and cohomological dimension for totally
  disconnected locally compact groups.
\newblock {\em Journal of Topology and Analysis}, 0(0):1--27, 0.

\bibitem[AM21]{AM21}
Shivam {Arora} and Eduardo {Mart{\'\i}nez-Pedroza}.
\newblock {Fixed Point Sets in Diagrammatically Reducible Complexes}.
\newblock {\em arXiv e-prints}, page arXiv:2107.01254, July 2021.

\bibitem[BH99]{BH99}
Martin~R. Bridson and Andr\'e Haefliger.
\newblock {\em Metric spaces of non-positive curvature}, volume 319 of {\em
  Grundlehren der Mathematischen Wissenschaften [Fundamental Principles of
  Mathematical Sciences]}.
\newblock Springer-Verlag, Berlin, 1999.

\bibitem[BLN01]{BLN01}
Noel Brady, Ian~J. Leary, and Brita E.~A. Nucinkis.
\newblock On algebraic and geometric dimensions for groups with torsion.
\newblock {\em J. London Math. Soc. (2)}, 64(2):489--500, 2001.

\bibitem[BM91]{BeMe91}
Mladen Bestvina and Geoffrey Mess.
\newblock The boundary of negatively curved groups.
\newblock {\em J. Amer. Math. Soc.}, 4(3):469--481, 1991.

\bibitem[Bow98]{Bow98}
B.~H. Bowditch.
\newblock Continuously many quasi-isometry classes of {$2$}-generator groups.
\newblock {\em Comment. Math. Helv.}, 73(2):232--236, 1998.

\bibitem[Bow12]{Bo12}
B.~H. Bowditch.
\newblock Relatively hyperbolic groups.
\newblock {\em Internat. J. Algebra Comput.}, 22(3):1250016, 66, 2012.

\bibitem[Bro87]{Br87}
Kenneth~S. Brown.
\newblock Finiteness properties of groups.
\newblock In {\em Proceedings of the {N}orthwestern conference on cohomology of
  groups ({E}vanston, {I}ll., 1985)}, volume~44, pages 45--75, 1987.

\bibitem[DJ91]{DJ91}
Michael~W. Davis and Tadeusz Januszkiewicz.
\newblock Hyperbolization of polyhedra.
\newblock {\em J. Differential Geom.}, 34(2):347--388, 1991.

\bibitem[Dra99]{D99}
A.~N. Dranishnikov.
\newblock Boundaries of {C}oxeter groups and simplicial complexes with given
  links.
\newblock {\em J. Pure Appl. Algebra}, 137(2):139--151, 1999.

\bibitem[Dun85]{Dun85}
M.~J. Dunwoody.
\newblock The accessibility of finitely presented groups.
\newblock {\em Invent. Math.}, 81(3):449--457, 1985.

\bibitem[EG57]{EG57}
Samuel Eilenberg and Tudor Ganea.
\newblock On the {L}usternik-{S}chnirelmann category of abstract groups.
\newblock {\em Ann. of Math. (2)}, 65:517--518, 1957.

\bibitem[FJ93]{FJ93}
F.~T. Farrell and L.~E. Jones.
\newblock Isomorphism conjectures in algebraic {$K$}-theory.
\newblock {\em J. Amer. Math. Soc.}, 6(2):249--297, 1993.

\bibitem[FL14]{FL14}
Martin~G. Fluch and Ian~J. Leary.
\newblock An {E}ilenberg-{G}anea phenomenon for actions with virtually cyclic
  stabilisers.
\newblock {\em Groups Geom. Dyn.}, 8(1):135--142, 2014.

\bibitem[FN13]{FN13}
Martin~G. Fluch and Brita E.~A. Nucinkis.
\newblock On the classifying space for the family of virtually cyclic subgroups
  for elementary amenable groups.
\newblock {\em Proc. Amer. Math. Soc.}, 141(11):3755--3769, 2013.

\bibitem[HMP14]{HaMa13}
Richard~Gaelan Hanlon and Eduardo Mart\'{\i}nez-Pedroza.
\newblock Lifting group actions, equivariant towers and subgroups of
  non-positively curved groups.
\newblock {\em Algebr. Geom. Topol.}, 14(5):2783--2808, 2014.

\bibitem[JPL06]{JPL06}
Daniel Juan-Pineda and Ian~J. Leary.
\newblock On classifying spaces for the family of virtually cyclic subgroups.
\newblock In {\em Recent developments in algebraic topology}, volume 407 of
  {\em Contemp. Math.}, pages 135--145. Amer. Math. Soc., Providence, RI, 2006.

\bibitem[KLS20]{KLS20}
Robert~P. Kropholler, Ian~J. Leary, and Ignat Soroko.
\newblock Uncountably many quasi-isometry classes of groups of type {$FP$}.
\newblock {\em Amer. J. Math.}, 142(6):1931--1944, 2020.

\bibitem[LM00]{LM00}
Wolfgang L{\"u}ck and David Meintrup.
\newblock On the universal space for group actions with compact isotropy.
\newblock In {\em Geometry and topology: {A}arhus (1998)}, volume 258 of {\em
  Contemp. Math.}, pages 293--305. Amer. Math. Soc., Providence, RI, 2000.

\bibitem[LS01]{LySc01}
Roger~C. Lyndon and Paul~E. Schupp.
\newblock {\em Combinatorial group theory}.
\newblock Classics in Mathematics. Springer-Verlag, Berlin, 2001.
\newblock Reprint of the 1977 edition.

\bibitem[LS21]{LASS21}
Porfirio~L. {Le{\'o}n {\'A}lvarez} and Luis~Jorge {S{\'a}nchez Salda{\~n}a}.
\newblock {Virtually abelian dimension for $3$-manifold groups}.
\newblock {\em arXiv e-prints}, page arXiv:2105.05905, May 2021.

\bibitem[L{\"u}c00]{Lu00}
Wolfgang L{\"u}ck.
\newblock The type of the classifying space for a family of subgroups.
\newblock {\em J. Pure Appl. Algebra}, 149(2):177--203, 2000.

\bibitem[L{\"u}c05]{Lu05}
Wolfgang L{\"u}ck.
\newblock Survey on classifying spaces for families of subgroups.
\newblock In {\em Infinite groups: geometric, combinatorial and dynamical
  aspects}, volume 248 of {\em Progr. Math.}, pages 269--322. Birkh\"auser,
  Basel, 2005.

\bibitem[LW12]{LW12}
Wolfgang L{\"u}ck and Michael Weiermann.
\newblock On the classifying space of the family of virtually cyclic subgroups.
\newblock {\em Pure Appl. Math. Q.}, 8(2):497--555, 2012.

\bibitem[MPW11a]{MPW11}
Eduardo Mart\'{\i}nez-Pedroza and Daniel~T. Wise.
\newblock Local quasiconvexity of groups acting on small cancellation
  complexes.
\newblock {\em J. Pure Appl. Algebra}, 215(10):2396--2405, 2011.

\bibitem[MPW11b]{MPW11b}
Eduardo Mart\'{\i}nez-Pedroza and Daniel~T. Wise.
\newblock Relative quasiconvexity using fine hyperbolic graphs.
\newblock {\em Algebr. Geom. Topol.}, 11(1):477--501, 2011.

\bibitem[MS20]{MPSS20}
Eduardo {Mart\'{\i}nez-Pedroza} and Luis~Jorge {S\'anchez Salda\~na}.
\newblock {Brown's criterion and classifying spaces for families}.
\newblock {\em {J. Pure Appl. Algebra}}, 224(10):16, 2020.
\newblock Id/No 106377.

\bibitem[MV03]{MV03}
Guido Mislin and Alain Valette.
\newblock {\em Proper group actions and the {B}aum-{C}onnes conjecture}.
\newblock Advanced Courses in Mathematics. CRM Barcelona. Birkh\"auser Verlag,
  Basel, 2003.

\bibitem[Osi06]{osin}
Denis~V. Osin.
\newblock Relatively hyperbolic groups: intrinsic geometry, algebraic
  properties, and algorithmic problems.
\newblock {\em Mem. Amer. Math. Soc.}, 179(843):vi+100, 2006.

\bibitem[Sey74]{Sey74}
Judith Elaine~Kaldenberg Seymour.
\newblock {\em Conjugate powers and unique roots in certain small cancellation
  groups}.
\newblock ProQuest LLC, Ann Arbor, MI, 1974.
\newblock Thesis (Ph.D.)--University of Illinois at Urbana-Champaign.

\bibitem[SSn20]{SS20}
Luis~Jorge S\'{a}nchez Salda\~{n}a.
\newblock Groups acting on trees and the {E}ilenberg-{G}anea problem for
  families.
\newblock {\em Proc. Amer. Math. Soc.}, 148(12):5469--5479, 2020.

\bibitem[Tru74]{Truf74}
Bernard Truffault.
\newblock Centralisateurs des \'{e}l\'{e}ments dans les groupes de
  {G}reendlinger.
\newblock {\em C. R. Acad. Sci. Paris S\'{e}r. A}, 279:317--319, 1974.

\end{thebibliography}
\end{document}